\documentclass{amsart}

\theoremstyle{plain}
\newtheorem{thm}{Theorem}[section]
\theoremstyle{definition}
\newtheorem{defn}[thm]{Definition}
\theoremstyle{definition}
\newtheorem{example}[thm]{Example}
\theoremstyle{plain}
\newtheorem{lem}[thm]{Lemma}
\theoremstyle{plain}
\newtheorem{prop}[thm]{Proposition}
\theoremstyle{remark}
\newtheorem{notation}[thm]{Notation}

\usepackage{mathbbol}

\usepackage[all,dvips,arc,curve,rotate]{xy}

\begin{document}

\title{$C^{*}$-Algebra Relations}

\author{Terry A. Loring}

\address{Department of Mathematics and Statistics, University of New Mexico,
Albuquerque, NM 87131, USA.}

\urladdr{http://www.math.unm.edu/\textasciitilde{}loring}

\begin{abstract}
We investigate relations on elements in $C^{*}$-algebras, including
$*$-polynomial relations, order relations and all relations that
correspond to universal $C^{*}$-algebras. We call these $C^{*}$-relations
and define them axiomatically. Within these are the compact $C^{*}$-relations,
which are those that determine universal $C^{*}$-algebras, and we
introduce the more flexible concept of a closed $C^{*}$-relation. 

In the case of a finite set of generators, we show that
closed $C^{*}$-relations
correspond to the zero-sets of elements in a free $\sigma$-$C^{*}$-algebra.
This provides a solid link between two of the previous theories on
relations in $C^{*}$-algebras.

Applications to lifting problems are briefly considered in the last
section.
\end{abstract}

\maketitle

\section{Introduction}

In the contexts of operator inequalities, lifting problems, $K$-theory
and universal $C^{*}$-algebras, the need arises for relations on
an element $x$ in a $C^{*}$-algebra $A$ that that are best described
in terms of the $\mathbf{M}_{2}(A).$ An example is the relation 
\[
0\leq\left[\begin{array}{cc}
\left|x\right| & x^{*}\\
x & \left|x\right|\end{array}\right]\leq1
\]
on $x.$  We also need relations such as
\[
\left[\begin{array}{cc}
x_{11} & x_{12}\\
x_{21} & x_{22}\end{array}\right]^{2}=\left[\begin{array}{cc}
x_{11} & x_{12}\\
x_{21} & x_{22}\end{array}\right]
\]
that give do not determine universal $C^{*}$-algebras. 

The variety of relations that arise in operator theory is impressive.
In \cite{LoringFromMatrixToOperator} we study questions about operators
that can be reduced to questions about matrices. The relations that
arise include 
\[
\alpha\leq e^{x+x^{*}}\leq\beta
\]
and
\[
\left\Vert y\sqrt{\left|x\right|}-\sqrt{\left|x\right|}y\right\Vert \leq\delta.
\]
This example-rich environment will support a general theory, a theory
of $C^{*}$-relations. 

Two existing theories are compelling: that of Phillips in
\cite[\S 1.3]{PhillipsProCstarAlg}
and that of Hadwin, Kaonga and Mathes in \cite[\S 6]{Hadwin-Kaonga-Mathes}.
The allowed class of relations is, for our purposes, too large in
first instance, too small in the second. Our compromise is an axiomatic
approach that is more restrictive than allowed by Phillips. In the
case of finitely many generators, we can show that a subset of these
relations are equivalent to relations in the same basic form as considered
by Hadwin et al.

The lack of free $C^{*}$-algebras forces us to consider pro-$C^{*}$-algebras.
For background on this class of $*$-algebras, see \cite{Fragoulopoulou}
or \cite{PhillipsProCstarAlg}.

Another name for a pro-$C^{*}$-algebra is locally-$C^{*}$-algebra.
A pro-$C^{*}$-algebra is a topological $*$-algebra whose topology
arises from, and is complete with respect to, a set of $C^{*}$-seminorms.
Those seminorms are not part of the object in this category. The
morphisms are all continuous $*$-homomorphisms.

This terminology is in conflict with Grothendieck's notion of a pro-category
(\cite[p.\  4]{AtiyahSegalCompletion}). The conflict is slight, as continuous
$*$-homomorphisms give rise to families of $*$-homomorphisms between
$C^{*}$-algebras, as in Lemma~\ref{lem:proMorphismPicture}.

When a pro-$C^{*}$-algebra has a topology described by a sequence
of $C^{*}$-seminorms, it is metrizable and called a $\sigma$-$C^{*}$-algebra. 

The free pro-$C^{*}$-algebras $\mathbb{F}\langle x_{1},\ldots,x_{n}\rangle$
are $\sigma$-$C^{*}$-algebras. They contain in a nice way the $*$-polynomials
in finitely many noncommuting variables. The elements of
$\mathbb{F}\langle x_{1},\ldots,x_{n}\rangle$
are the noncommutative functions of Hadwin, Kaonga and Mathes, and
their zero sets provide a rich class of $C^{*}$-algebra relations.

There is a lot of confusion in the definition of a
relation for $C^{*}$-algebras,
mostly arising from the fact that free $C^{*}$-algebras do not exist
(except on zero generators). We cannot simply define the relations
as being elements of the free object that have been set to zero. The
free object we can access is in the wrong category, and is not easily
understood as it arises from completion with respect to a uncomputable
sequence of seminorms.

We can define a relation as a ``statement about elements in a $C^{*}$-algebra,''
but must take care. It is easy to have hidden ideas of what statements
are allowed. We only need to know the class of functions 
$f:\mathcal{X}\rightarrow A$
that are to be representations of a relation, so we work directly
with categories whose objects are functions from sets to $C^{*}$-algebras.

The statement 
\[
0\leq a_{1}\leq a_{2}\leq1
\]
is to be thought of as shorthand for the category whose objects are
functions
\[
f:\{ x_{1},x_{2}\}\rightarrow A
\]
for which
\[
0\leq f(x_{1})\leq f(x_{2})\leq1
\]
and whose morphisms are intertwining $*$-homomorphisms. The desired
universal representation
\[
\iota: \{ x_{1},x_{2}\} \rightarrow 
C^{*}\left\langle
x_{1},x_{2}\left|\,0\leq x_{1}\leq x_{2}\leq1\right.
\right\rangle 
\]
is the initial object in that category.

\section{$C^{*}$-Algebra Relations 
\label{sec:C^{*}-Algebra-Relations}}

We identify within a general class of relations those that correspond
to universal $C^{*}$-algebras.

\begin{defn}
Given a set $\mathcal{X},$ the \emph{null} $C^{*}$\emph{-relation}
on $\mathcal{X}$ is the category $\mathcal{F}_{\mathcal{X}}$ with
objects of the form $(j,A),$ where $A$ is a $C^{*}$-algebra and
$j:\mathcal{X}\rightarrow A$ is a function.  The morphisms from
$(j,A)$ to $(k,B)$ all $*$-homomorphisms $\varphi:A\rightarrow B$
for which $\varphi\circ j=k.$
\end{defn}

Given any nonempty set $\Lambda$ and $C^{*}$-algebras $A_{\lambda}$ for
$\lambda\in\Lambda,$ we use one of
\[
\prod_{\lambda\in\Lambda}A_{\lambda}
\mbox{ or }
\prod_{\lambda\in\Lambda}^{C^{*}}A_{\lambda}
\]
to denote the $C^{*}$-algebra of families 
$\left\langle a_{\lambda}\right\rangle _{\lambda\in\Lambda}$
that are bounded in norm and have $a_{\lambda}\in A_{\lambda}.$ 

\begin{defn}
Given a set $\mathcal{X},$ a \emph{$C^{*}$-relation on $\mathcal{X}$}
is a full subcategory $\mathcal{R}$ of $\mathcal{F}_{X}$ of such
that:
\begin{description}
\item [{C1}] 
the unique map $\mathcal{X}\rightarrow\{0\}$ is an object;
\item [{C2}] 
if $\varphi:A\hookrightarrow B$ is an injective $*$-homomorphism
and $f:\mathcal{X}\rightarrow A$ is a function, then
\[
f\mbox{ is an object }\Leftarrow\varphi\circ f\mbox{ is an object;}
\]
\item [{C3}] 
if $\varphi:A\rightarrow B$ is a $*$-homomorphism and 
$f:\mathcal{X}\rightarrow A$
is a function, then
\[
f\mbox{ is an object }\Rightarrow\varphi\circ f\mbox{ is an object;}
\]
\item [{C4f}] 
if $f_{j}:\mathcal{X}\rightarrow A_{j}$ is an object for
$1\leq j\leq n$ then 
\[
\prod f_{j}:\mathcal{X}\rightarrow\prod_{j=1}^{n}A_{j}
\]
is an object.
\end{description}
\end{defn}

The admissible relations defined in \cite{PhillipsProCstarAlg} are
only required to satisfy \textbf{C3} in the case where $\varphi$
is a surjection. Such a relation can be extended to a $C^{*}$-relation
by adding in any push-forward by an inclusion.

The \emph{intersection} of two or more $C^{*}$-relations on the same
set $\mathcal{X}$ will be the full subcategory whose objects are the 
$f:\mathcal{X} \rightarrow A$ that are representations of all the
given relations.  This intersection is again a $C^{*}$-relation.

We will generally not mention the morphisms as they are determined
by the objects.

\begin{defn}
The \emph{$C^{*}$}-relation $\mathcal{R}$ on \emph{$\mathcal{X}$}
is called \emph{compact} if
\begin{description}
\item [{C4}] 
for any nonempty set $\Lambda,$ if 
$f_{\lambda}:\mathcal{X}\rightarrow A_{\lambda}$
is an object for all $\lambda\in\Lambda$ then 
\[
\prod f_{\lambda}:\mathcal{X} \rightarrow
\prod_{\lambda\in\Lambda}A_{\lambda}
\]
is an object.
\end{description}
\end{defn}

\begin{example}
Let $\mathcal{R}$ be the subcategory of $\mathcal{F}_{\emptyset}$
whose only object is the unique function from $\emptyset$ to the
zero $C^{*}$-algebra. This is satisfies \textbf{C1}, \textbf{C2}
and \textbf{C4} but only the weaker form of \textbf{C3} where $\varphi$
is only allowed to be a surjection. 
\end{example}

Usually we will have a statement that determines the objects in a
$C^{*}$-relation. We will call this statement a $C^{*}$-relation
and the objects in the associated category \emph{representations of
the relation.} If we start with $\mathcal{R}$ we can use
\[
f:\mathcal{X}\rightarrow A\mbox{ is an object in }\mathcal{R}
\]
as a relation whose representations are the objects in $\mathcal{R}.$
For this reason, we generally call an object a representation.

\begin{example}
If $p$ is a noncommutative $*$-polynomial in $n$ variables with
zero constant term then
\[
p(x_1,\ldots,x_n)=0
\]
is a $C^*$-relation.
\end{example}

\begin{example}
The $C^*$-relation
\[
x^*x - x = 0
\]
is compact, since $x^*x=x$ implies $x^*=x$ and so $x$ is a projection, and
so has norm at most one.
\end{example}

\begin{example}
The $C^*$-relation
\[
x^2 - x = 0
\]
is a not compact, as idempotents can have any norm.
\end{example}

\begin{example}
Consider the relation determined by the equation
\[
xy - yx - 1 = 0,
\]
where if $x$ and $y$ are in $A$ then this relation holds if $A$
is unital and $xy-yx$ equals the unit in $A.$
If we allow the case $1=0$ in the zero $C^*$-algebra then
\textbf{C3} will fail.  If we exclude this case, then
\textbf{C1} will fail.  Either way, we do not obtain a $C^{*}$-relation.
\end{example}

For any $C^{*}$-algebra $A,$ let
\[
\mbox{rep}_{\mathcal{R}}(\mathcal{X},A)
=
\left\{ 
f:\mathcal{X}\rightarrow A
\left|
\ f\mbox{ is a representation of }\mathcal{R}
\right.
\right\} .
\]

\begin{defn}
If $\mathcal{X}$ is a set and $\mathcal{R}$ is a $C^{*}$-relation
on $\mathcal{X}$ then a function $\iota:\mathcal{X}\rightarrow U$
from $\mathcal{X}$ to a $C^{*}$-algebra $U$ is \emph{universal
for $\mathcal{R}$} if:
\begin{description}
\item [{U1}] 
given a $C^{*}$-algebra $A,$ if 
$\varphi:U\rightarrow A$ is a $*$-homomorphism
then $\varphi\circ\iota:\mathcal{X}\rightarrow A$
is a representation of $\mathcal{R};$ 
\item [{U2}] 
given a $C^{*}$-algebra $A,$ if a function $f:\mathcal{X}\rightarrow A$
is a representation of $\mathcal{R}$ then there is a unique $*$-homomorphism
$\varphi:U\rightarrow A$ so that $f=\varphi\circ\iota.$
\end{description}
\end{defn}

It should be clear that $\iota$ and $U$ are unique up to isomorphism.
Notice that $\iota$ must be a representation. The definition of a
universal representation is summarized by the bijection
\[
\hom(U,A)\rightarrow\mbox{rep}_{\mathcal{R}}(\mathcal{X},A)
\]
defined by $\varphi\mapsto\varphi\circ\iota.$

Notice that \textbf{U1} is absent in \cite[\S 1.3]{PhillipsProCstarAlg}.
See \cite[Definition 1.2]{Blackadar-shape-theory}.

Various versions of Theorem~\ref{thm:repsIFFuniversalExists} can
be found in \cite[\S 1.4]{Hadwin-Kaonga-Mathes},
\cite[\S 3.1]{Loring-lifting-perturbing}
and \cite[Proposition 1.3.6]{PhillipsProCstarAlg}.
The proof here uses the same techniques
as Hadwin and Ma in \cite[\S 2]{HadwinMaFreeProducts}.

\begin{thm}
\label{thm:repsIFFuniversalExists} 
If $\mathcal{R}$ is $C^{*}$-relation
on $\mathcal{X}$ then $\mathcal{R}$ is compact if and only if
there exists a universal representation for $\mathcal{R}.$ 
\end{thm}

\begin{proof}
Assume such a universal representation $f:\mathcal{X}\rightarrow U$
exists. We need to verify \textbf{C4.}

Suppose $\Lambda$ is a nonempty set and 
$f_{\lambda}:\mathcal{X}\rightarrow A_{\lambda}$
is a representation for each $\lambda\in\Lambda.$ For each $\lambda$
we know there exists a 
$*$-homomorphism $\varphi_{\lambda}:U\rightarrow A_{\lambda}$
with $f_{\lambda}=\varphi_{\lambda}\circ\iota.$ Since
\[
\prod f_{\lambda}=\left(\prod\varphi_{\lambda}\right)\circ\iota
\]
we have proven \textbf{C4}.

As to the converse, assume $\mathcal{R}$ is a compact $C^{*}$-relation
on $\mathcal{X}.$ 

Let $S_{1}$ be a set such that every $C^{*}$-algebra generated by
a set no larger than $\mathcal{X}$ has cardinality at most the cardinality
of $S_{1}.$ Let $S_{2}$ be the set of all $C^{*}$-algebras whose
underlying set is a subset of $S_{1}.$ Let $S_{3}$ be the set of
all functions from $\mathcal{X}$ to a $C^{*}$-algebra in $S_{2}.$
Let $S_{4}$ be the set containing every function $f:\mathcal{X}\rightarrow A$
in $S_{3}$ whose image $f(\mathcal{X})$ generates $A$ and so that
$f$ is a representation in $\mathcal{R}.$ Let these representations
be indexed as $f_{\lambda}:\mathcal{X}\rightarrow A_{\lambda}$ for
$\lambda$ in a set \emph{}$\Lambda.$ 

Given any representation $g:\mathcal{X}\rightarrow B,$ by \textbf{C2}
we know that by corestricting we can factor $g$ as $g=\alpha\circ g_{0}$
where $g_{0}:\mathcal{X}\rightarrow B_{0}$ has image that generates
$B_{0}$ and $\alpha:B_{0}\rightarrow B$ is an inclusion. There will
be an isomorphism $\beta:B_{0}\rightarrow B_{1}$ for some $B_{1}$
in $S_{2}.$ Let $g_{1}:\mathcal{X}\rightarrow B_{1}$ be defined
as $g_{1}=\beta\circ g_{0}.$ This will be a representation by \textbf{C3},
with generating image, and so $g_{1}=f_{\lambda}$ and $B_{1}=A_{\lambda}$
for some $\lambda$ in $\Lambda.$ Thus $g$ factors as 
$g=\gamma\circ f_{\lambda}$
where $g:A_{\lambda}\rightarrow B$ is the injective $*$-homomorphism
$g=\alpha\circ\beta^{-1}.$

To summarize the last paragraph: every representation $g$ in $\mathcal{R}$
can be factored as
$g=\varphi\circ f_{\lambda}$ where $\varphi:A_{\lambda}\rightarrow B$
is an injective $*$-homomorphism. 

By \textbf{C1} there is a representation, so we know $\Lambda\neq\emptyset.$

Let
\[
f=\prod_{\lambda\in\Lambda}f_{\lambda}   :  
\mathcal{X}\rightarrow\prod_{\lambda\in\Lambda}A_{\lambda}.
\]
This is well defined and a representation by \textbf{C4}. Let $U$
denote the $C^{*}$-algebra generated by the image of $f$ and let
$\iota:\mathcal{X}\rightarrow U$ be the corestriction of $f.$ The
inclusion of $U$ in the product we call $\eta,$ so $f=\eta\circ\iota.$ 

Suppose $\varphi:U\rightarrow A$ is a $*$-homomorphism.  Since
$\iota$ is a representation, \textbf{C3} tells us that $\varphi\circ\iota$
is also a representation. We have meet the first requirement on $U.$ 

Suppose $B$ is a $C^{*}$-algebra and that 
a function $g:\mathcal{X}\rightarrow B$
is a representation in $\mathcal{R}.$ We can factor $g$ as 
$g=\varphi\circ f_{\lambda_{0}}$
where $\varphi:A_{\lambda_{0}}\rightarrow B$ is an 
injective $*$-homomorphism.
Let $p_{\lambda_{0}}$ denote the coordinate projection
\[
p_{\lambda_{0}} : 
\prod_{\lambda\in\Lambda}A_{\lambda}\rightarrow A_{\lambda_{0}}.
\]
Define $\psi:U\rightarrow B$ as the $*$-homomorphism 
$\psi=\varphi\circ p_{\lambda_{0}}\circ\eta.$
Then
\[
\psi\circ\iota = \varphi\circ p_{\lambda_{0}}\circ f
= \varphi\circ f_{\lambda_{0}} = g.
\]
Since $\iota$ has range that generates $U,$ the $*$-homomorphism
$\psi$ is the unique one satisfying $\psi\circ\iota=g.$
\end{proof}
If $\mathcal{R}$ is a $C^{*}$-relation with universal representation
$\iota:\mathcal{X}\rightarrow U$ then we call $U$ the \emph{universal
$C^{*}$-algebra for $\mathcal{R}$} and use for notation 
$U=C^{*}\left\langle \mathcal{X}\left|\mathcal{R}\right.\right\rangle .$
Sometimes we will use $\iota_{\mathcal{R}}$ in place of the generic
$\iota.$

\begin{example}
There is one free $C^{*}$-algebra, namely 
$C^{*}\left\langle 
\emptyset\left|\mathcal{F}_{\emptyset}\right.
\right\rangle ,$
which is just $\{0\}.$
\end{example}

\begin{example}
For any $C^{*}$-algebra $A,$
\[
C^{*}\left\langle 
A\left|A\rightarrow B\mbox{ is a $*$-homomorphism}\right.
\right\rangle \cong A.
\]
That is, if we let $\mathcal{R}_{A}$ be the full 
subcategory of $\mathcal{F}_{A}$
with objects $f:A\rightarrow B$ that are $*$-homomorphisms, then
$A$ is isomorphic to 
$C^{*}\left\langle A\left|\mathcal{R}_{A}\right.\right\rangle .$
\end{example}

Neither the zero sets of noncommutative polynomials, not even $a^{2}=0,$
nor basic order relations like $a\leq b$ are compact. The story must
continue, and that means leaving our familiar category.

\section{\label{sec:pro-Cstar-Relations} Relations in Pro-$C^{*}$-algebras}

For any pro-$C^{*}$-algebra $A,$ let $S(A)$ denote the set of all
continuous $C^{*}$-seminorms on $A.$ For $p$ in $S(A)$ we have
the $C^{*}$-algebra $A_{p}=A/\ker(p)$ and the surjection 
$\pi_{p}:A\rightarrow A_{p}.$
For $q\geq p$ we have also surjections $\pi_{q,p}:A_{q}\rightarrow A_{p}.$
If $S\subseteq S(A)$ is cofinal then 
${\displaystyle A=\lim_{\longleftarrow}A_{p}}$
where $p$ ranges over $S.$ 

Starting from an inverse system of $C^{*}$-algebras, 
$\rho_{\lambda^{\prime},\lambda}:A_{\lambda^{\prime}}\rightarrow A_{\lambda}$
for $\lambda\preceq\lambda^{\prime}$ in $\Lambda,$ we can take the
inverse limit and get a pro-$C^{*}$-algebra 
${\displaystyle A=\lim_{\longleftarrow}A_{\lambda}}.$
However, the induced $*$-homomorphisms 
$\rho_{\lambda}:A\rightarrow A_{\lambda}$
may fail to be surjective. However, if $\Lambda=\mathbb{N}$ then
the $\rho_{\lambda}$ are always surjections. For proofs of these
facts, see \cite[\S 1]{PhillipsProCstarAlg} and
\cite[\S 5]{PhillipsInverseLimits}.

\begin{lem}
\label{lem:inductiveSystemsReallySetsOfSeminorm} 
Suppose 
${\displaystyle A=\lim_{\longleftarrow}A_{\lambda}}$
is a pro-$C^{*}$-algebra and $\rho_{\lambda}:A\rightarrow A_{\lambda}$
is a surjection for all $\lambda$ in $\Lambda.$ There is an order-preserving,
cofinal map $\gamma:\Lambda\rightarrow S(A)$ and there are isomorphisms
$\varphi_{\lambda}:A_{\gamma(\lambda)}\rightarrow A_{\lambda}$ so
that $\rho_{\lambda}=\varphi_{\lambda}\circ\pi_{\gamma(\lambda)}.$
\end{lem}

\begin{proof}
Simply define $\gamma(\lambda)(a)=\|\rho_{\lambda}(a)\|.$ There is
clearly an injective $*$-homomorphism $\varphi_{\lambda}$ defined
by
\[
\varphi_{\lambda}(a+\ker\gamma(\lambda))=\rho_{\lambda}(a)
\]
and it is onto because $\rho_{\lambda}$ is assumed to be onto. If
$\lambda\preceq\lambda^{\prime}$ then $\rho_{\lambda^{\prime},\lambda}$
is norm decreasing, which is easily seen to imply 
$\gamma(\lambda)\leq\gamma(\lambda^{\prime}).$
The inverse limit topology on $A$ is determined by the $\gamma(\lambda)$
and so $\gamma(\Lambda)$ must be cofinal. 
\end{proof}

\begin{lem}
\label{lem:proMorphismEmbeddings}
Suppose $A$ and $B$ 
are pro-$C^{*}$-algebras
and that $T\subseteq S(B)$ is cofinal. If $\varphi:A\rightarrow B$
is a $*$-homomorphism that is a homeomorphism onto its image then
there is a cofinal function $\theta:T\rightarrow S(A)$ and injective
$*$-homomorphisms $\varphi_{p}:A_{\theta(p)}\hookrightarrow B_{p}$
so that, for all $p$ in $T,$ we have 
$\pi_{p}\circ\varphi=\varphi_{p}\circ\pi_{\theta(p)}.$ 
\end{lem}

\begin{proof}
For any $p$ in $T$ we know that $p\circ\varphi$ is in $S(A),$
so we define $\theta(p)=p\circ\varphi.$ Since $a\in\ker(\theta(p))$
implies 
\[
\|\pi_{p}\circ\varphi(a)\|=p(\varphi(a))=0
\]
we find that $\varphi$ induces a $*$-homomorphism $\varphi_{p}$
from $A_{\theta(p)}$ to $B_{p}$ with 
$\pi_{p}\circ\varphi=\varphi_{p}\circ\pi_{\theta(p)}.$
It is injective since
\[
\|\varphi_{p}(\pi_{\theta(p)}(a))\|
  =  \|\pi_{p}(\varphi(a))\|
  =  p(\varphi(a)) 
  =   \theta(p)(a)
  =  \|\pi_{\theta(p)}(a)\|.
\]

We wish to show that $\theta(T)$ is cofinal. For $p$ in $T$ let
\[
B(p,\epsilon)=\left\{ b\in B\left|\, \strut p(b)<\epsilon\right.\right\} 
\]
and define $B(q,\epsilon)$ similarly for $q$ in $S(A).$ These sets
form neighborhood bases at the respective origins.

Suppose $q$ is in $S(A).$ Since $\varphi$ is open, there is an
$\epsilon>0$ and a $p$ in $T$ so that
\[
B(p,\epsilon)\subseteq\varphi(B(q,1)).
\]
For $a$ in $A,$
\[
\theta(p)(a)<\epsilon
\implies
\exists a_{1}\in A \mbox{ s.t.\,}
q(a_1)<1\mbox{ and }\varphi(a_{1})=\varphi(a)
\]
and, since $\varphi$ is one-to-one, 
\[
\theta(p)(a)<\epsilon\implies q(a)<1.
\]
Standard facts about $C^{*}$-algebras show that this implies 
$q\leq\theta(p).$
\end{proof}

\begin{lem}
\label{lem:proMorphismPicture}
Suppose $A$ and $B$ are pro-$C^{*}$-algebras
and that $S\subseteq S(A)$ and $T\subseteq S(B)$ are cofinal. If
$\varphi:A\rightarrow B$ is a continuous $*$-homomorphism then there
is a function $\theta:T\rightarrow S$ and there are $*$-homomorphisms
$\varphi_{p}:A_{\theta(p)}\rightarrow B_{p}$ so that 
$\pi_{p}\circ\varphi=\varphi_{p}\circ\pi_{\theta(p)}$
for all $p$ in $T.$
\end{lem}

\begin{proof}
For any $p$ in $T,$ we know that $p\circ\varphi$ is in $S(A),$
so choose $\theta(p)\in S$ with $\theta(p)\geq p\circ\varphi.$
Since $a\in\ker(\theta(p))$ implies 
\[
\|\pi_{p}\circ\varphi(a)\|=p(\varphi(a))=0
\]
we find that $\varphi$ induces a $*$-homomorphism $\varphi_{p}$
from $A_{\theta(p)}$ to $B_{p}$ with 
$\pi_{p}\circ\varphi=\varphi_{p}\circ\pi_{\theta(p)}.$ 
\end{proof}

In the last two lemmas, the function $\theta$ and the maps 
$\varphi_{p}:A_{\theta(p)}\rightarrow B_{p}$
are a morphism in the sense of Grothendieck 
(\cite{AtiyahSegalCompletion})
between the pro-objects $(A,\{ A_{p}\},\{\pi_{p}\})$ 
and $(B,\{ B_{p}\},\{\pi_{p}\}).$
That is, one can show that if $p\leq p^{\prime}$ and $q\geq\theta(p)$
and $q\geq\theta(p')$ then 
\[
\varphi_{p}\circ\pi_{q,\theta(p)}
=
\pi_{p^{\prime},p}\circ\varphi_{p^{\prime}}\circ\pi_{q,\theta(p^{\prime})}.
\]

\begin{lem}
\label{lem:mapsProCstarToCstarFactor} 
Suppose ${\displaystyle A=\lim_{\longleftarrow}A_{\lambda}}$
is a pro-$C^{*}$-algebra and $\rho_{\lambda}:A\rightarrow A_{\lambda}$
is a surjection for all $\lambda$ in $\Lambda.$ Suppose $B$ is
a $C^{*}$-algebra. If $\varphi:A\rightarrow B$ is a continuous $*$-homomorphism
then there exists $\lambda$ in $\Lambda$ and 
$\varphi^{\prime}:A_{\lambda}\rightarrow B$
so that $\varphi=\varphi^{\prime}\circ\rho_{\lambda}.$
\end{lem}

\begin{proof}
Lemma~\ref{lem:inductiveSystemsReallySetsOfSeminorm} reduces this
to a special case of Lemma~\ref{lem:proMorphismPicture}.
\end{proof}

\begin{lem}
Suppose $\mathcal{R}$ is a $C^{*}$-relation on $\mathcal{X}.$ Suppose
$f:\mathcal{X}\rightarrow A$ is a function and $A$ is a pro-$C^{*}$-algebra.
If $\pi_{p}\circ f$ is a representation of $\mathcal{R}$ in $A_{p}$
for all $p$ in a cofinal set $S$ in $S(A)$ then $\varphi\circ f$
is a representation of $\mathcal{R}$ for every continuous $*$-homomorphisms
$\varphi$ from $A$ to a $C^{*}$-algebra.
\end{lem}

\begin{proof}
Composition with a $*$-isomorphism preserves representations of $\mathcal{R},$
so it suffices to show $\pi_{p}\circ f$ is a representation for any
$p$ in $S(A).$ Since $S$ is cofinal, we know $\pi_{q}\circ f$
is a representation for some $q\geq p.$ Therefore
\[
\pi_{p}\circ f  =  \pi_{q,p}\circ\pi_{q}\circ f
\]
is a representation.
\end{proof}

Given $A_\lambda$ a pro-$C^*$-algebra for each $\lambda$ in
a set $\Lambda,$ we denote the $*$-algebra of all families
$\langle a_\lambda \rangle$ with $a_\lambda \in A_\lambda$ by
\[
A=\prod_{\lambda\in\Lambda}^{\textnormal{pro}C^{*}}A_{\lambda},
\]
with projection maps $\rho_{\lambda}:A\rightarrow A_{\lambda}.$
This becomes a pro-$C^*$-algebra if we endow it with the
product topology.

\begin{lem}
If $A_\lambda$ is a family of pro-$C^*$-algebras and
\[
A=\prod_{\lambda\in\Lambda}^{\textnormal{pro}C^{*}}A_{\lambda},
\]
then the seminorms of the form
\[
\max\left(
p_{1}\circ\rho_{\lambda_{1}},\ldots,p_{n}\circ\rho_{\lambda_{n}}
\right)
\]
 for $p_{j}$ in $S\left(A_{\lambda_{j}}\right)$ are cofinal in $S(A).$
\end{lem}

\begin{proof}
A collection of $C^{*}$-seminorms on $A$ that is closed under the
pairwise max operation is cofinal if and only if it determines the
topology on $A.$ In this case, the topology is component-wise convergence,
and the seminorms $p\circ\rho_{\lambda},$ for $p\in S\left(A_{\lambda}\right),$
determine the topology.
\end{proof}

Suppose $\mathcal{X}$ is any set. For each 
$l:\mathcal{X}\rightarrow[0,\infty)$
define
\[
F_{l}\langle\mathcal{X}\rangle
=
C^{*}\left\langle
\mathcal{X}\left|\forall x\in\mathcal{X},\,\| x\|\leq l(x)\right.
\right\rangle 
\]
with $\iota_{l}$ the universal representation. Consider also the
$*$-algebra of $*$-polynomials in noncommuting variables 
\[
\left\{ x,x^{*}\left|x\in\mathcal{X}\right.\right\} 
\]
 (the $x^{*}$ being some symbols not in $\mathcal{X}$) hereby denoted
$\mathbb{C}\left[\mathcal{X}\cup\mathcal{X}^{*}\right].$

Lemma~\ref{lem:StarPolysEmbed} is by Goodearl and Menal
\cite[Proposition 2.2]{GoodearlMenalRFDandFree}.
The proof is only a little modified from theirs.

\begin{lem}
\label{lem:StarPolysEmbed}For any $l>0$ the canonical $*$-homomorphism
\[
\mathbb{C}\left[\mathcal{X}\cup\mathcal{X}^{*}\right]\rightarrow
C^{*}
\left\langle
\mathcal{X}\left|\|\,\strut x\|\leq l(x)\right.
\right\rangle 
\]
is one-to-one.
\end{lem}

\begin{proof}
For two nonzero choices for $l$ we get isomorphic $C^{*}$-algebras.
It is then easy to reduce to the case $l(x)=2$ for all $x.$

Let $U$ denote the full group $C^{*}$-algebra of the free group
generated by two copies of $\mathcal{X}.$ Let the two disjoint copies
of $x\in\mathcal{X}$ be $\dot{x}$ and $\bar{x}.$ In terms of generators
and relations in the category of unital $C^{*}$-algebras, 
\[
U=
C_{1}^{*}\left\langle
\dot{\mathcal{X}}\cup\bar{\mathcal{X}}
\left|\,\strut
\mbox{each }\dot{x}\mbox{ and }\bar{x}\mbox{ is unitary}
\right.
\right\rangle .
\]
We know that the group algebra embeds in $U,$ and so it is safe to
drop the inclusion map from our notation. Define
\[
\varphi:
\mathbb{C}\left[\mathcal{X}\cup\mathcal{X}^{*}\right]\rightarrow U
\]
by $\varphi(x)=\dot{x}+\bar{x}.$ Notice
$\varphi(x^{*})=\dot{x}^{-1}+\bar{x}^{-1}.$
Given a $*$-polynomial $p$ of degree $n,$ consider the terms in
$\varphi(p)$ that are in the alternating pattern 
``$\dot{\;}\;\bar{\;}\;\dot{\;}\;\bar{\;}\;\dot{\;}\;\bar{\;}.$''
Any simplifying that happens in $\varphi(p)$ will not involve these
terms, so $\varphi(p)=0$ implies that all monomials have coefficient
zero. This means $\varphi$ is injective, and the result follows.

To illustrate the argument based on the pattern of decorations, suppose
$\mathcal{X}=\{ x\}$ and
\[
p=x^{*}x+2xx^{*}+3x.
\]
Then
\[
\varphi(p)
=
\left(\dot{x}^{-1}\bar{x}+2\dot{x}\bar{x}^{-1}\right) +
\left(\bar{x}^{-1}\dot{x}+2\bar{x}\dot{x}^{-1}\right)
+3\dot{x}+3\bar{x}+6
\]
and so the dot-dash terms of length two reflect the coefficients of
the terms of length two in $p.$
\end{proof}

There are surjections between these ``free'' $C^{*}$-algebras. If
$l\geq l^{\prime}$ then sending $x$ to $x$ determines
\[
\gamma_{l,l^{\prime}} :
F_{l}\langle\mathcal{X}\rangle \rightarrow
F_{l^{\prime}}\langle\mathcal{X}\rangle.
\]
Finally let 
${\displaystyle \mathbb{F}\langle\mathcal{X}\rangle
=\lim_{\longleftarrow}F_{l}\langle\mathcal{X}\rangle}$
and $\iota:\mathcal{X}\rightarrow\mathbb{F}\langle\mathcal{X}\rangle$
be defined so that $\iota(x)$ corresponds to the coherent family
$\left\langle \iota_{l}(x)\right\rangle _{l}.$ There are the obvious
$*$-homomorphisms
$\gamma_{l} :
\mathbb{F}\langle\mathcal{X}\rangle \rightarrow
F_{l}\langle\mathcal{X}\rangle.$  These are in fact surjections, as
each generator determines a coherent family that is then sent to
the copy of that generator in $F_{l}\langle\mathcal{X}\rangle.$
Notice $\iota(\mathcal{X})$ algebraically generates a dense copy
of $\mathbb{C}\left[\mathcal{X}\cup\mathcal{X}^{*}\right].$ 

\begin{thm}
In the category of pro-$C^{*}$-algebras and continuous $*$-homomorphisms,
$\iota:\mathcal{X}\rightarrow\mathbb{F}\langle\mathcal{X}\rangle$
is free. 
\end{thm}

\begin{proof}
First suppose $A$ is a $C^{*}$-algebra. 
For any function $f:\mathcal{X}\rightarrow A$
we can set $l(x)=\| f(x)\|$ and there is a 
$*$-homomorphism 
$\varphi_{l}:F_{l}\langle\mathcal{X}\rangle\rightarrow A$
sending $\iota_{l}(x)$ to $f(a).$ Then $\varphi_{l}\circ\gamma_{l}$
is a continuous $*$-homomorphism that sends $\iota(x)$ to $f(a).$
This is the unique such map since $\iota(\mathcal{X})$ generates
$\mathbb{F}\langle\mathcal{X}\rangle.$ 

Suppose $A$ is a pro-$C^{*}$-algebra and $f:\mathcal{X}\rightarrow A$
is a function. For each $p$ in $S(A)$ there
is a unique continuous $*$-homomorphism
$\varphi_{p}:\mathbb{F}\langle\mathcal{X}\rangle\rightarrow A_{p}$
for which $\varphi_{p}\circ\iota=\pi_{p}\circ f.$ Since
\[
\pi_{p,p^{\prime}}\circ\varphi_{p}\circ\iota 
  =   \pi_{p,p^{\prime}}\circ\pi_{p}\circ f 
   =   \pi_{p^{\prime}}\circ f
\]
we can conclude 
$\pi_{p,p^{\prime}}\circ\varphi_{p}=\varphi_{p^{\prime}}.$
This means there is a continuous 
$*$-homomorphism 
$\varphi:\mathbb{F}\langle\mathcal{X}\rangle\rightarrow A$
so that $\pi_{p}\circ\varphi=\varphi_{p}.$ Therefore
\[
  \pi_{p}(\varphi(\iota(x)))  
= \varphi_{p}(\iota(x)) 
= \pi_{p}(f(x))
\]
and so $\varphi(\iota(x))=f(x).$ 

The uniqueness of $\varphi$ again follows from the fact that 
$\iota(\mathcal{X})$
generates $\mathbb{F}\langle\mathcal{X}\rangle.$
\end{proof}

\begin{lem}
The \emph{pro-}$C^{*}$-algebra 
$\mathbb{F}\langle\mathcal{X}\rangle$
is a $\sigma$-$C^{*}$-algebra if and only if 
$\mathcal{X}$ is finite.
\end{lem}

\begin{proof}
Suppose $\mathcal{X}$ is the finite set 
$\{ x_{1},\ldots,x_{n}\}.$
The functions $l_{k}$ defined by $l_{k}(x_{j})=k$ are cofinal among
all functions from $\mathcal{X}$ to $[0,\infty)$. 
Therefore $\mathbb{F}\langle x_{1},\ldots,x_{n}\rangle$
is an inverse limit of a sequence of $C^{*}$-algebras,
\[
\mathbb{F}\langle x_{1},\ldots,x_{n}\rangle
= \lim_{\longleftarrow} 
C^{*}\left\langle
x_{1},\ldots,x_{n}\left| \,\strut
\left\Vert
x_{1}\right\Vert \leq k,\ldots,\left\Vert x_{n}\right\Vert \leq k
\right.
\right\rangle .
\]

For the converse it suffices to show that 
$\mathbb{F}\langle x_{1},x_{2},\ldots\rangle$
is not a $\sigma$-$C^{*}$-algebra. 

Suppose $p_{1},p_{2},\ldots$ is an increasing sequence of $C^{*}$-seminorms
determining the topology of $\mathbb{F}\langle x_{1},x_{2},\ldots\rangle.$
By passing to a subsequence, we are able to
assume $p_{n}(\iota(x_{n})) \neq 0$
for all $n.$  Define
\[
\alpha_{k}=\min_{n\leq k}\left(kp_{n}(\iota(x_{k}))\right)^{-1}
\]
and $y_{n}=\alpha_{n}\iota(x_{n}).$ For $k\geq n$ we have 
$p_{n}(y_{k})\leq\frac{1}{k}.$
Therefore $\lim_{k\rightarrow\infty}y_{k}=0.$ Take any 
sequence $a_{k}$
in $B(\mathbb{H})$ so that $\| a_{k}\|=\alpha_{k}^{-1}.$ There is
a continuous $*$-homomorphism
\[
\varphi:\mathbb{F}\langle x_{1},x_{2},\ldots\rangle
\rightarrow B(\mathbb{H})
\]
with $\varphi(\iota(x_{k}))=a_{k}.$ 
This means $\alpha_{k}a_{k}$
converges to zero, contradicting the fact that 
$\|\alpha_{k}a_{k}\|$
has norm $1.$
\end{proof}

\begin{defn}
Given a set $X,$ the \emph{null pro-$C^{*}$-relation on} $\mathcal{X}$
is the category $\mathcal{F}_{\mathcal{X}}^{\textnormal{pro}C^{*}}$
whose objects are of the form $(j,A),$ where $A$ is a pro-$C^{*}$-algebra
and $j:\mathcal{X}\rightarrow A$ is a function from $\mathcal{X}$
to (the underlying set of) $A.$ As morphisms from $(j,A)$ to $(k,B)$
it has all continuous $*$-homomorphisms $\varphi:A\rightarrow B$
for which $\varphi\circ j=k.$
\end{defn}

\begin{defn}
Given a set $\mathcal{X},$ a \emph{pro-$C^{*}$relation on $\mathcal{X}$}
is full subcategory 
$\mathcal{R}$ of 
$\mathcal{F}_{\mathcal{X}}^{\textnormal{pro}C^{*}}$
such that:
\begin{description}
\item [{PC1}] 
the unique map $\mathcal{X}\rightarrow\{0\}$ is an object;
\item [{PC2}] 
if $\varphi:A\hookrightarrow B$ is the inclusion of a closed
$*$-subalgebra of a pro-$C^{*}$-algebra $B$ and if 
$f:\mathcal{X}\rightarrow A$
is a function, then
\[
f\mbox{ is an object }\Leftarrow\varphi\circ f\mbox{ is an object;}
\]
\item [{PC3}] 
if $\varphi:A\rightarrow B$ is a 
continuous $*$-homomorphism,
and if $f:\mathcal{X}\rightarrow A$ is a function, then
\[
f\mbox{ is an object } \Rightarrow
\varphi\circ f\mbox{ is an object;}
\]
\item [{PC4}] 
if $\Lambda$ is a nonempty set, and if 
$f_{\lambda}:\mathcal{X}\rightarrow A_{\lambda}$
is an object for each $\lambda\in\Lambda,$ then 
\[
\prod f_{\lambda} :
\mathcal{X}
\rightarrow
\prod_{\lambda\in\Lambda}^{\textnormal{pro}C^{*}}A_{\lambda}
\]
is an object.
\end{description}
\end{defn}

Again we will conflate statements with categories and representations
with objects.

\begin{defn}
Suppose $\mathcal{X}$ is a set and $\mathcal{R}$ is 
a pro-$C^{*}$-relation
on $\mathcal{X}.$ A function $\iota:\mathcal{X}\rightarrow U$ from
$\mathcal{X}$ to a pro-$C^{*}$-algebra $U$ is \emph{universal for
$\mathcal{R}$} if:
\begin{description}
\item [{PU1}] 
given a pro-$C^{*}$-algebra $A,$ if $\varphi:U\rightarrow A$
is a continuous $*$-homomorphism  then 
$\varphi\circ\iota:\mathcal{X}\rightarrow A$
is a representation of $\mathcal{R};$ 
\item [{PU2}] 
given a pro-$C^{*}$-algebra $A,$ if a 
function $f:\mathcal{X}\rightarrow A$
is a representation in $\mathcal{R}$ then there is a 
unique $*$-homomorphism
$\varphi:U\rightarrow A$ so that 
$f=\varphi\circ\iota.$
\end{description}
\end{defn}

It should be clear that $\iota$ and $U$ are unique, up to isomorphism.
Also notice that $\iota$ must be a representation.

The definition of a universal representation is again summarized by
the bijection 
\[
\hom(U,A)\rightarrow\mbox{rep}_{\mathcal{R}}(\mathcal{X},A)
\]
defined by $\varphi\mapsto\varphi\circ\iota,$ but now for $A$ any
pro-$C^{*}$-algebra and $\hom(\mbox{--},\mbox{--})$ meaning
the set of continuous $*$-homomorphisms.

\begin{thm}
If $\mathcal{R}$ is a pro-$C^{*}$-relation on $\mathcal{X}$ then
there exists a universal representation for $\mathcal{R}.$ 
\end{thm}

\begin{proof}
Suppose $g:\mathcal{X}\rightarrow A$ is a representation 
of $\mathcal{R}.$
Let $B$ be the closed $*$-algebra generated by 
$g(\mathcal{X}).$
There is a continuous $*$-homomorphism 
$\varphi:\mathbb{F}\langle\mathcal{X}\rangle\rightarrow B$
so that $\varphi(\iota(x))=g(x).$ By \textbf{PC2}, 
we can corestrict
$g$ to a representation $f_{1}:\mathcal{X}\rightarrow B.$ Let $\kappa$
be the inclusion of $B$ in $A,$ so $\kappa\circ f_{1}=g.$ There
is an open, continuous {*}-algebra isomorphism
\[
\psi : 
\overline{\mathbb{F}\langle\mathcal{X}\rangle/\ker(\varphi)}\rightarrow B
\]
where the completion is with respect to the seminorms 
\[
\varphi(y)+\ker(\varphi) \mapsto
p\left(\varphi(y)\right)
\quad(\mbox{for } p\in S(B).)
\]
By \textbf{PC3}, $f_{2}=\psi^{-1}\circ f_{1}$ is a representation
and $f=\kappa\circ\psi\circ f_{2}.$

The algebraic quotients of $\mathbb{F}\langle\mathcal{X}\rangle$
by closed, two-sided self-adjoint ideals form a set. The collection
of all $C^{*}$-seminorms on each quotient is a set, and so the collection
of all possible completions of quotients of 
$\mathbb{F}\langle\mathcal{X}\rangle$
is a set. Therefore, we can index by a set $\Lambda$ all representation
into these particular pro-$C^{*}$-algebras 
$f_{\lambda}:\mathcal{X}\rightarrow A_{\lambda}$
so that a generic representation $g$ as above factors as 
$g=\gamma\circ f_{\lambda}$
for some continuous $*$-homomorphism $\gamma.$ 

By \textbf{PC1} there are representations, so we know
$\Lambda\neq\emptyset.$

Let
\[
f=\prod_{\lambda\in\Lambda}f_{\lambda} :
\mathcal{X} \rightarrow 
\prod_{\lambda\in\Lambda}^{\textnormal{pro}C^{*}}A_{\lambda}
\]
This is well-defined and a representation by \textbf{PC4}. Let $U$
denote the pro-$C^{*}$-algebra generated by the image of $f$ and let
$\iota:\mathcal{X}\rightarrow U$ be the corestriction of $f.$ The
inclusion of $U$ in the product we call $\eta,$ so 
$f=\eta\circ\iota.$ 

The proof that $\iota$ is universal for 
$\mathcal{R}$ is similar
to the argument given in the proof of 
Theorem~\ref{thm:repsIFFuniversalExists}.
\end{proof}

For notation, the universal pro-$C^{*}$ algebra will be
\[
\textnormal{pro}C^{*}
\left\langle 
\mathcal{X}\left|\mathcal{R}\right.
\right\rangle .
\]

If $A$ is a pro-$C^{*}$-algebra and $\mathcal{R}_{A}$ is defined
on the set $A$ with $\varphi:A\rightarrow B$ considered a representation
if and only if it is a continuous $*$-homomorphism, then $\mathcal{R}_{A}$
is a pro-$C^{*}$-relation and $A$ is isomorphic to 
$\textnormal{pro}C^{*}
\left\langle 
A\left|\mathcal{R}_{A}\right.
\right\rangle .
$
This can easily be made a bit more general.

\begin{lem}
\label{lem:proC*RelationsFromProCstarAlgebra} 
Suppose $f:\mathcal{X}\rightarrow A$
is a function whose image generates the 
pro-$C^{*}$-algebra $A.$
Let $\mathcal{R}_{f}$ be the full subcategory of
$\mathcal{F}_{\mathcal{X}}^{\textnormal{pro}C^{*}}$
for which
\[
\textnormal{rep}_{\mathcal{R}_{f}}(B) =
\left\{ \varphi\circ f\left|\varphi\in\hom(A,B)\right.\right\} .
\]
Then $\mathcal{R}$ is a pro-$C^{*}$-relation and 
\[
\textnormal{pro}C^{*}\left\langle
\mathcal{X}\left|\mathcal{R}_{A_{f}}\right.
\right\rangle 
\cong A,
\]
where the isomorphism sends $\iota(x)$ to $f(x).$
\end{lem}

\begin{proof}
We know the zero function $A\rightarrow\{0\}$ is in 
$\hom(A,\{0\})$
and so the zero function $\mathcal{X}\rightarrow\{0\}$ 
is a representation.

Suppose $g:\mathcal{X}\rightarrow B$ is a function and 
$\psi:B\hookrightarrow C$
is an embedding of a closed $*$-subalgebra and
$\psi\circ g$ is
a representation. Then $\psi\circ g=\varphi\circ f$ for 
some $\varphi$
in $ \hom(A,C).$ Thus $\varphi(f(\mathcal{X}))\subseteq B$ 
and so
$\varphi(A)\subseteq B$ and $\varphi=\psi\circ\varphi_{0}$ 
for some $\varphi_{0}$ in $\hom(A,B)$ and 
\[
\psi\circ\varphi_{0}\circ f=\psi\circ g.
\]
Since $\psi$ is injective, $\varphi_{0}\circ f=g$ and $g$ is a
representation.

If $g:\mathcal{X}\rightarrow B$ is a representation 
and $\psi$ is
in $\hom(B,C)$ then $g=\varphi\circ f$ for some $\varphi$ in $\hom(A,B).$
Therefore $\psi\circ g=\psi\circ\varphi\circ f$ is a representation.

Suppose $g_{\lambda}:\mathcal{X}\rightarrow B_{\lambda}$ 
is a representation
for all $\lambda\in\Lambda.$ Then $g_{\lambda}=\varphi_{\lambda}\circ f$
for some $\varphi_{\lambda}$ in $\hom(A,B_{\lambda}).$ Then
\[
\prod g_{\lambda}=\left(\prod\varphi_{\lambda}\right)\circ f
\]
is a representation.

For the second statement, we need to show that 
$f:\mathcal{X}\rightarrow A$
is universal. But that says there is a bijection
\[
\hom(A,B) \rightarrow 
\mbox{rep}_{\mathcal{R}_{f}}(\mathcal{X},B)
\]
 defined by $\varphi\mapsto\varphi\circ f,$ and
 this is true by definition.
\end{proof}

\begin{lem}
\label{pro:LCstarRepsClosedUnderInvLimits}
Every pro-$C^{*}$-relation is closed under inverse limits.
\end{lem}

\begin{proof}
Suppose $\mathcal{R}$ is a pro-$C^{*}$-relation on $\mathcal{X}.$
Suppose we have an inverse system. That is $A_{\lambda}$ 
is a pro-$C^{*}$-algebra
for each $\lambda$ in a directed set $\Lambda$ and there are bonding
maps $\rho_{\lambda,\mu}:A_{\lambda}\rightarrow A_{\mu}$ that are
continuous $*$-homomorphisms whenever $\mu\preceq\lambda.$ Then
the limit can be constructed as
\[
\lim_{\longleftarrow}A_{\lambda} =
\left\{ \left.\left\langle a_{\lambda}\right\rangle
\in\prod_{\lambda\in\Lambda}^{\textnormal{pro}C^{*}}A_{\lambda}
\right|\rho_{\lambda,\mu}(a_{\lambda}) = 
a_{\mu}\mbox{ if }\mu\preceq\lambda\right\} 
\]
and $\rho_{\lambda}:A\rightarrow A_{\lambda}$ defined by
$
\rho_{\lambda}\left(\left\langle a_{\alpha}\right\rangle \right)
=
a_{\lambda}.
$

Given $f_{\lambda}:\mathcal{X}\rightarrow A_{\lambda},$ 
representations
that are coherent in the sense that 
$\rho_{\lambda,\mu}\circ f_{\lambda}=f_{\mu}$
wherever $\mu\preceq\lambda,$ we have a function 
$f:\mathcal{X}\rightarrow A$
define by corestricting the product,
\[
f(x) =
\left\langle f_{\lambda}(x)\right\rangle
\in\lim_{\longleftarrow}A_{\lambda}
\]
and this is a representation by \textbf{PC2} and \textbf{PC4}.
\end{proof}

\begin{prop}
If $\mathcal{R}$ is a
pro-$C^{*}$-relation, then its restriction to $C^{*}$-algebras is
a $C^{*}$-relation. If two pro-$C^{*}$-relations on the same set
have the same restriction to $C^{*}$-algebras then they are equal.
\end{prop}

\begin{proof}
The first statement is clear, since the pro-$C^{*}$ product of a
finite number of $C^{*}$-algebras equals the $C^{*}$ product.

As to the second, every pro-$C^{*}$-algebra is the inverse limit
of $C^{*}$-algebras, so Lemma~\ref{pro:LCstarRepsClosedUnderInvLimits}
applies.
\end{proof}

\begin{prop}
\label{pro:ExtendingCstarRelationsToProCstar} Suppose $\mathcal{R}$
is a $C^{*}$-relation on $\mathcal{X}.$ If we define 
$\hat{\mathcal{R}}$
as the full subcategory of 
$\mathcal{F}_{\mathcal{X}}^{\textnormal{pro}C^{*}},$
where $f:\mathcal{X}\rightarrow A$ is an object if $\pi_{p}\circ f$
is a representation of $\mathcal{R}$ for all $p$ in $S(A),$ then
$\hat{\mathcal{R}}$ is a pro-$C^{*}$-relation extending $\mathcal{R}.$
\end{prop}

\begin{proof}
Since quotients take representations to representations, 
$\hat{\mathcal{R}}$
extends $\mathcal{R}.$ Notice also that
$f:\mathcal{X}\rightarrow A$
must be a representation for $\hat{\mathcal{R}}$ 
if $\pi_{p}\circ f$
is a representation of $\mathcal{R}$ for all $p$ in a cofinal set
in $S(A).$

Since $0:\mathcal{X}\rightarrow\{0\}$ is a representation 
in $\mathcal{R}$
it is also a representation in $\hat{\mathcal{R}}.$

Suppose $\varphi:A\hookrightarrow B$ is the inclusion of a closed
$*$-subalgebra of a pro-$C^{*}$-algebra $B$ and 
$f:\mathcal{X}\rightarrow A$
is a function for which $\varphi\circ f$ is a representation 
of $\hat{\mathcal{R}.}$  By Lemma~\ref{lem:proMorphismEmbeddings} 
there is a cofinal function 
$\theta:S(B)\rightarrow S(A)$ and injective
$*$-homomorphisms 
$\varphi_{p}:A_{\theta(p)}\hookrightarrow B_{p}$
so that $\pi_{p}\circ\varphi=\varphi_{p}\circ\pi_{\theta(p)}$ for
all $p$ in $S(B).$ We know that
\[
\pi_{p}\circ\varphi\circ f =
\varphi_{p}\circ\pi_{\theta(p)}\circ f
\]
is a representation of $\mathcal{R},$ and since $\varphi_{p}$ is
injective, also that $\pi_{\theta(p)}\circ f$ is a representation
of $\mathcal{R}.$ Since the image of $\theta$ is cofinal in $S(A),$
we conclude $f$ is a representation of $\mathcal{R}.$

Suppose $\varphi:A\rightarrow B$ is a continuous $*$-homomorphism
and $f:\mathcal{X}\rightarrow A$ is 
a representation of $\hat{\mathcal{R}}.$
By Lemma~\ref{lem:proMorphismPicture} there is a
function $\theta:S(B)\rightarrow S(A)$ and $*$-homomorphisms
$\varphi_{p}:A_{\theta(p)}\rightarrow B_{p}$ so 
that $\pi_{p}\circ\varphi=\varphi_{p}\circ\pi_{\theta(p)}$
for all $p$ in $S(B).$ Since $f$ is a 
representation of $\hat{\mathcal{R}},$
we know $\pi_{\theta(p)}\circ f$ is a 
representation of $\mathcal{R},$
and so 
\[
\pi_{p}\circ\varphi\circ f=\varphi_{p}\circ\pi_{\theta(p)}\circ f
\]
 is a representation of $\mathcal{R}.$ This 
 being true for all $p$
in $S(B),$ we conclude $\varphi\circ f$ is a 
representation of $\hat{\mathcal{R}}.$

Suppose $f_{\lambda}:\mathcal{X}\rightarrow A_{\lambda}$ 
is a representation
of $\hat{\mathcal{R}}$ for each $\lambda$ in a 
nonempty set $\Lambda.$
To show $f=\prod f_{\lambda}$ is a representation of 
$\hat{\mathcal{R}},$
it suffices to show $\pi_{q}\circ f$ is a representation for 
\[
q=\max\left(p_{1}\circ\rho_{\lambda_{1}},\ldots,p_{n} \circ 
\rho_{\lambda_{n}}\right).
\]
Let
\[
A=\prod_{\lambda\in\Lambda}^{\textnormal{pro}C^{*}}A_{\lambda}.
\]
Consider the continuous $*$-homomorphism
\[
\gamma : A \rightarrow
\left(A_{\lambda_{1}}\right)_{p_{n}} 
\oplus \cdots\oplus
\left(A_{\lambda_{n}}\right)_{p_{n}}
\]
 defined as 
\[
\gamma =
\pi_{p_{1}}\circ\rho_{\lambda_{1}}
\oplus\cdots\oplus
\pi_{p_{n}}\circ\rho_{\lambda_{n}}.
\]
This corresponds to the seminorm $q,$ as
\begin{align*}
\left\Vert 
\gamma\left(\left\langle a_{\lambda}\right\rangle _{\lambda}\right)
\right\Vert  
& = 
\left\Vert 
\pi_{p_{1}}\left(a_{\lambda_{1}}\right)\oplus
\cdots
\oplus\pi_{p_{n}}\left(a_{\lambda_{n}}\right)
\right\Vert \\
 & =  
 \max\left(
 \left\Vert \pi_{p_{1}}\left(a_{\lambda_{1}}\right)\right\Vert ,
 \ldots,\left\Vert \pi_{p_{1}}\left(a_{\lambda_{1}}\right)\right\Vert 
 \right)\\
 & =  \max\left(
 p_{1}\left(a_{\lambda_{1}}\right),\ldots,
 p_{n}\left(a_{\lambda_{1}}\right)
 \right)
 \end{align*}
and so we have a $*$-isomorphism
\[
\psi  : A_{q}\rightarrow\left(A_{\lambda_{1}}\right)_{p_{n}}
\oplus\cdots\oplus
\left(A_{\lambda_{n}}\right)_{p_{n}}
\]
 satisfying $\psi\circ\pi_{q}=\gamma.$ Finally 
 \begin{align*}
\pi_{q}\circ f 
& =  \psi^{-1}\circ\gamma\circ f\\
 & =  \psi^{-1}\circ\left(\pi_{p_{1}}\circ\rho_{\lambda_{1}}\circ f
 \oplus\cdots\oplus
 \pi_{p_{n}}\circ\rho_{\lambda_{n}}\circ f\right)\\
 & =  \psi^{-1}\circ\left(\pi_{p_{1}}\circ f_{\lambda_{1}}
 \oplus\cdots\oplus
 \pi_{p_{n}}\circ f_{\lambda_{n}}\right)
 \end{align*}
which means $\pi_{q}\circ f$ is a representation of $\mathcal{R}.$
\end{proof}

\section{Pushouts of Pro-$C^{*}$-algebras}

Recall that a diagram of pro-$C^{*}$-algebras and continuous $*$-homomorphisms
\[
\xymatrix{
C \ar[r] ^{\theta_2} \ar[d] ^{\theta_1} 	& 	B \ar[d] ^{\iota_2} \\
A \ar[r] ^{\iota_1} & D
}
\]
is a pushout (and $D$ an amalgamated free product)
if $\varphi\mapsto(\varphi\circ\iota_{1},\varphi\circ\iota_{2})$
determines a bijection
\[
\hom(D,E) \rightarrow
\left\{ \left.\left(\varphi_{1},\varphi_{2}\right)\in\hom(A,E)
\times
\hom(B,E)
\,\right|\,\strut
\varphi_{1}\circ\theta_{1} = 
\varphi_{2}\circ\theta_{2}\right\} .
\]

By the usual category theory result we know that pushouts must be
unique.

Lemma~\ref{lem:pushoutsExist} extends
\cite[Proposition 1.5.3(1)]{PhillipsProCstarAlg},
showing pushouts exist in full generality.

\begin{lem}
\label{lem:pushoutsExist}
Suppose $A,$ $B$ and $C$ 
are pro-$C^{*}$-algebras
and that 
$\theta_{1}:C\rightarrow A$
and $\theta_{2}:C\rightarrow B$ are continuous $*$-homomorphisms.
Assume $A$ and $B$ are disjoint.
Define $\mathcal{R}$ to have as representations each 
function $f:A\cup B\rightarrow E$
such that $f|_{A}:A\rightarrow E$ and 
$f|_{B}:B\rightarrow E$ are
continuous $*$-homomorphisms and 
$f\circ\theta_{1}=f\circ\theta_{2}.$
Then $\mathcal{R}$ is a pro-$C^{*}$-relation. The
diagram 
\[
\xymatrix{
C \ar[r] ^{\theta_2} \ar[d] ^{\theta_1} 	
	& 	B \ar[d] ^{\iota|_B}  \\
A \ar[r] ^(0.3){\iota|_A}	
		&	C_{\textnormal{pro}}^{*}
		\left\langle A\cup B\left|\mathcal{R}\right.\right\rangle 
}
\]is a pushout. 
\end{lem}

\begin{proof}
The proof is routine.
\end{proof}

\begin{lem}
\label{lem:altDescPushouts} Suppose 
\[
\xymatrix{
C \ar[r] ^{\theta_2} \ar[d] ^{\theta_1} 	& 	B \ar[d] ^{\iota_2} \\
A \ar[r] ^{\iota_1} & D
}
\]
a diagram
of pro-$C^{*}$-algebras and continuous $*$-homomorphisms. This is
a pushout if and only if $\iota_{1}(A)\cup\iota_{2}(B)$ generates
$D$ and for every pair
\[
(\varphi_{1},\varphi_{2})\in\hom(A,E)\times\hom(B,E)
\]
such that $\varphi_{1}\circ\theta_{1} 
=\varphi_{2}\circ\theta_{2}$
there exists $\varphi$ in $\hom(D,E)$ with 
$\varphi\circ\iota_{j}=\varphi_{j}.$ 
\end{lem}

\begin{proof}
Without loss of generality, $A$ and $B$ are disjoint.

Pushouts are unique. If the diagram is a pushout then up to isomorphism
$D$ is given by generators $A\cup B$ and the relations as in 
Lemma~\ref{lem:pushoutsExist}.
Therefore
\[
\iota(A\cup B)=\iota_{1}(A)\cup\iota_{2}(B)
\]
must generate.

For the converse, we are given the existence of $\varphi$ for compatible
$\varphi_{1}$ and $\varphi_{2}$ and need only show uniqueness. However,
if $\iota_{1}(A)\cup\iota_{2}(B)$ generates, then $\varphi$ is uniquely
determined by $\varphi(\iota_{1}(a))=\varphi_{1}(a)$ and 
$\varphi(\iota_{2}(b))=\varphi_{2}(b).$
\end{proof}

\begin{lem}
\label{lem:CstarAlgsDeterminePushouts} Suppose 
\[
\xymatrix{
C \ar[r] ^{\theta_2} \ar[d] ^{\theta_1} 	& 	B \ar[d] ^{\iota_2} \\
A \ar[r] ^{\iota_1} & D
}
\]
a diagram of pro-$C^{*}$-algebras and continuous $*$-homomorphisms.
The diagram is a pushout if for every $C^{*}$-algebra $E$ and every
pair
\[
(\varphi_{1},\varphi_{2})\in\hom(A,E)\times\hom(B,E)
\]
such that $\varphi_{1}\circ\theta_{1}
=
\varphi_{2}\circ\theta_{2}$
there exists a unique $\varphi$ in $\hom(D,E)$ 
with $\varphi\circ\iota_{j}=\varphi_{j}.$ 
\end{lem}

\begin{proof}
This follows easily using the universal properties of pushouts and
inverse limits.
\end{proof}

\section{Pushouts in two categories}

First a look at an easy example of a pushout diagram in the category
of $C^{*}$-algebras. Then a method to create pushout diagrams in
the pro-$C^{*}$ category out of a sequence of pushouts in the $C^{*}$
category.

\begin{lem}
\label{lem:pushoutOfCstarSurjections}
Consider the commutative diagram
of $C^{*}$-algebras and $*$-homomorphisms 
\[
\xymatrix{
C \ar[r] ^{\beta} \ar[d] ^{\alpha}
			& B  \ar[d] ^{\gamma}
			\\
A \ar[r] ^{\delta}
			& X 
}
\]If
$\alpha$ and $\beta$ are onto and the square is a pushout then:
\begin{enumerate}
\item 
$\gamma$ and $\delta$ are surjections;
\item 
$\alpha(\ker(\beta))=\ker(\delta);$ 
\item 
given $a$ in $A$ and $b$ in $B$ with $\delta(a)=\gamma(b),$ there
exists $c$ in $C$ with $\alpha(c)=a$ and $\beta(c)=b.$
\end{enumerate}
\end{lem}

\begin{proof}
Without loss of generality, $B=C/J$ and $A=C/K$ for some ideals
$J$ and $K$ of $C.$ Since 
\[
\xymatrix{
C \ar[r] \ar[d] & C/J \ar[d] \\
C/K \ar[r] & C/(J+K)
}
\]is a pushout, and
pushouts are unique, we can also assume
\[
X=C/(J+K).\]
That shows (1).

Notice (2) is a special case of (3).

As to (3), we can assume we have $c$ and $c'$ in $C$ with $c-c'$
in $J+K.$ There are elements $j$ in $J$ and 
$k$ in $K$ with $c-k=c^{\prime}+j.$
Taking $c^{\prime\prime}=c-k$ we have $c^{\prime\prime}$ in $C$
with $c^{\prime\prime}+K=c+K$ and 
$c^{\prime\prime}+J=c^{\prime}+J.$
\end{proof}

\begin{thm}
\label{thm:inverseLimitsOfPushouts} 
Suppose 
\[
\xymatrix{
A_{n+1}  \ar[d] ^{\alpha_{n+1,n}} \ar[r] ^{\rho_{n+1}}
	& B_{n+1} \ar[d] ^{\beta_{n+1,n}}\\
A_n \ar[r] ^{\rho_{n}} 
	& B_n
}
\]
is a pushout in the category of $C^{*}$-algebras for all $n.$ Let
${\displaystyle A=\lim_{\longleftarrow}A_{n}}$ and
${\displaystyle B=\lim_{\longleftarrow}B_{n}}$
with associated maps $\alpha_{n}:A\rightarrow A_{n}$
and $\beta_{n}:B\rightarrow B_{n}.$
Define $\rho:A\rightarrow B$ by 
$\beta_{n}\circ\rho=\rho_{n}\circ\alpha_{n}.$
\begin{enumerate}
\item
If $\alpha_{n+1,n}$ and $\beta_{n+1,n}$ are surjective for all $n$
then the diagram 
\[
\xymatrix{
A  \ar[d] ^{\alpha_{n}} \ar[r] ^{\rho}
	& B \ar[d] ^{\beta_{n}}\\
A_n \ar[r] ^{\rho_{n}} 
	& B_n
}
\]
is a pushout in the category of pro-$C^{*}$-algebras.
\item 
If $\alpha_{n+1,n},$  $\beta_{n+1,n}$ and $\rho_{n}$ are
surjective for all $n$ then $\rho$ is a surjection.
\end{enumerate}
\end{thm}

\begin{proof}
(1)
It suffices to show that 
\[
\xymatrix{
A  \ar[d] ^{\alpha_{1}} \ar[r] ^{\rho}
	& B \ar[d] ^{\beta_{1}}\\
A_1 \ar[r] ^{\rho_{1}} 
	& B_1
}
\]is a pushout. By Lemma~\ref{lem:CstarAlgsDeterminePushouts} 
we need
only consider a $C^{*}$-algebra $E$ and 
$\varphi:A_{1}\rightarrow E$
and $\psi:B\rightarrow E$ such that 
$\varphi\circ\alpha_{1}=\psi\circ\rho.$
By Lemma~\ref{lem:mapsProCstarToCstarFactor} there is some $n$
and a map $\psi_{n}:B_{n}\rightarrow E$ so that 
$\psi=\psi_{n}\circ\beta_{n}.$
We have
\begin{align*}
\varphi\circ\alpha_{n-1,1}\circ\alpha_{n,n-1}\circ\alpha_{n} 
&=  \varphi\circ\alpha_{1}
 =  \psi\circ\rho
 =  \psi_{n}\circ\beta_{n}\circ\rho\\
 =  \psi_{n}\circ\rho_{n}\circ\alpha_{n},
\end{align*}
and since $\alpha_{n}$ is onto,
\[
\varphi\circ\alpha_{n-1,1}\circ\alpha_{n,n-1} =
\psi_{n}\circ\rho_{n}.
\]
The pushout property of the square involving $\rho_{n}$ 
and $\alpha_{n,n-1}$
tells us there is a $\psi_{n-1}:B_{n-1}\rightarrow E$ 
so that $\psi_{n}=\psi_{n-1}\circ\beta_{n,n-1}.$
Thus $\psi=\psi_{n-1}\circ\beta_{n-1}$ and we are where 
we were before,
but with $n$ decreased by one.

By induction, there is a continuous $*$-homomorphism 
$\psi_{1}:B_{1}\rightarrow E$
with $\psi=\psi_{1}\circ\beta_{1}.$ Also
\[
\varphi\circ\alpha_{1} 
 =  \psi\circ\rho
 =  \psi_{1}\circ\beta_{1}\circ\rho
 =  \psi_{1}\circ\rho_{1}\circ\alpha_{1}
\]
and $\alpha_{1}$ is onto so $\varphi=\psi_{1}\circ\rho_{1}.$ That
takes care of existence.

As to uniqueness, notice that $\rho_{1}(A_{1})$ equals $B_{1}$ so
the equation $\varphi=\psi_{1}\circ\rho_{1}$ makes $\varphi$ unique.

(2)
Given a coherent sequence $b_1, b_2, \ldots$ in $B_1, B_2, \ldots,$
we choose any $a_1$ with $\rho_1(a_1) = b_1.$  Now we repeatedly apply
Lemma~\ref{lem:pushoutOfCstarSurjections}
to find a coherent sequence $a_1, a_2, \ldots$ that is mapped to
$b_1, b_2, \ldots,$  proving the surjectivity of $\rho.$
\end{proof}

\section{Closed Relations}

\begin{defn}
For a set $\mathcal{X},$ 
and given functions 
$f_{\lambda}:\mathcal{X}\rightarrow A_{\lambda}$
into $C^{*}$-algebras $A_{\lambda}$ for each $\lambda$ is a nonempty
set $\Lambda,$ if 
\[
\sup_{\lambda}\left\Vert f_{\lambda}(x)\right\Vert <\infty
\]
for all $x$ then we call $\langle f_{\lambda}\rangle$ 
a \emph{bounded}
family of functions and define
\[
\prod f_{\lambda} :
\mathcal{X}\rightarrow\prod_{\lambda\in\Lambda}^{C^{*}}A_{\lambda}
\]
by
\[
\prod f_{\lambda}(x) =
\left\langle f_{\lambda}(x)\right\rangle _{\lambda\in\Lambda}.
\]

\end{defn}

\begin{defn}
A \emph{$C^{*}$}-relation on \emph{$\mathcal{X}$} is called \emph{closed}
if
\begin{description}
\item [{C4b}] 
if $\Lambda$ is a nonempty set, and if 
$f_{\lambda}:\mathcal{X}\rightarrow A_{\lambda}$
form a bounded family of objects, then
\[
\prod f_{\lambda} :
\mathcal{X}\rightarrow\prod_{\lambda\in\Lambda}A_{\lambda}
\]
is an object.
\end{description}
\end{defn}

Of course, compact implies closed. The intersection 
of a closed $C^{*}$-relation
with a compact $C^{*}$-relation is compact. An arbitrary intersection
of closed $C^{*}$-relations is closed.

Next we offer a sweepingly general functional calculus, as considered
in \cite[\S 4]{Hadwin-Kaonga-Mathes}.

\begin{defn}
\label{def:NCfunction} If $g$ is an 
element of $\mathbb{F}\langle x_{1},\ldots,x_{n}\rangle$
then we can define $g(a_{1},\ldots,a_{n})$ for 
$a_{j}\in A,$ where
$A$ is a pro-$C^{*}$-algebra, by 
\[
g(a_{1},\ldots,a_{n})=\varphi(g)
\]
where 
\[
\varphi:\mathbb{F}\langle x_{1},\ldots,x_{n}\rangle\rightarrow A
\]
is the unique continuous $*$-homomorphism 
defined by $\varphi(x_{j})=a_{j}.$
For example, if
\[
g=\sqrt{\iota(x_{1})^{*}\iota(x_{1})}+\iota(x_{2})
\]
then
\[
g(a_{1},a_{2})=\sqrt{a_{1}^{*}a_{1}}+a_{2}.
\]
This is clearly natural.
\end{defn}

\begin{thm}
\label{thm:freeElementsMakeClosedRelations} 
If $g$ is an element
of $\mathbb{F}\langle x_{1},\ldots,x_{n}\rangle$ 
then 
\[
g(x_{1},\ldots,x_{n})=
0\]
is a  closed $C^{*}$-relation.
\end{thm}

\begin{proof}
The only $*$-homomorphism from 
$\mathbb{F}\langle x_{1},\ldots,x_{n}\rangle$
to $\{0\}$ is the zero map $\zeta$, and so $\zeta(g)=0$ and so
the zero map from $\{ x_{1},\ldots,x_{2}\}$ to $\{0\}$ 
is a representation.

If $\varphi:A\rightarrow B$ is an injective $*$-homomorphism, and
if $f:\mathcal{X}\rightarrow A$ is a 
function so that $\varphi\circ f$
is a representation, then
\[
\varphi(g(f(x_{1}),\ldots,f(x_{n}))) = 
g(\varphi(f(x_{1})),\ldots,\varphi(f(x_{n})))
= 0
\]
so 
\[
g(f(x_{1}),\ldots,f(x_{n}))=0
\]
and $f$ is also a representation.

If $\varphi:A\rightarrow B$ is a 
$*$-homomorphism, and
if $f:\mathcal{X}\rightarrow A$ is a representation, then
\[
g(\varphi(f(x_{1})),\ldots,\varphi(f(x_{n}))) =
\varphi(g(f(x_{1}),\ldots,f(x_{n})))=0
\]
and so $\varphi\circ f$ is a representation.

Suppose $\Lambda$ is a nonempty set and that 
$f_{\lambda}:\mathcal{X}\rightarrow A_{\lambda}$
form a bounded family of relations. Let 
\[
f=\prod f_{\lambda} :
\mathcal{X}\rightarrow\prod_{\lambda\in\Lambda}^{C^{*}}A_{\lambda}.
\]
Let
\[
\varphi_{\lambda} :
\mathbb{F}\langle x_{1},\ldots,x_{n}\rangle\rightarrow A_{\lambda}
\]
and 
\[
\Phi :
\mathbb{F}\langle x_{1},\ldots,x_{n}\rangle \rightarrow
\prod_{\lambda\in\Lambda}^{C^{*}}A_{\lambda}
\]
be the associated continuous $*$-homomorphisms. Let 
$\rho_{\lambda}$
be the coordinate morphism, so that 
$\rho_{\lambda}\circ\Phi=\varphi_{\lambda}.$
In particular,
\[
\rho_{\lambda}\circ\Phi(g)=\varphi_{\lambda}(g)=0
\]
and so 
\[
g(f(x_{1}),\ldots,f(x_{n}))=\Phi(g)=0.
\]
 Therefore $f$ is a representation.
\end{proof}

Not all closed relations are best described by 
setting an element of $\mathbb{F}\langle x_{1},\ldots,x_{n}\rangle$ 
to zero. 

\begin{example}
If $p$ is a noncommutative $*$-polynomial in $n$ variables with
zero constant term and $C$ is a positive constant then
\[
\left\| p(x_1,\ldots,x_n) \right\| \leq C
\]
is a closed $C^*$-relation.
\end{example}

\begin{example}
The  inequality
\[
\| x \|< 1
\]
is a $C^*$-relation that is not closed.
\end{example}

\begin{example}
Let $\mathcal{X}$ denote a copy of $[0,1],$
\[
\mathcal{X} = \left \{ x_t \left |\, t \in [0,1] \right.  \right \}
\]
The statement
\[
  t \mapsto x_t \mbox{ is continuous }
\]
is a $C^*$-relation that is not closed.  This example and variations
are discussed in \cite[\S 1.3]{PhillipsProCstarAlg}.
\end{example}

We want something like a universal representation, but technically
not a representation since the function $\iota$ might not take $\mathcal{X}$
into a $C^{*}$-algebra.

\begin{defn}
If $\mathcal{X}$ is a set and $\mathcal{R}$ is a full subcategory
of $\mathcal{F}_{X},$ then a function $\iota:\mathcal{X}\rightarrow U$
from $\mathcal{X}$ to a pro-$C^{*}$-algebra $U$ is \emph{ubiquitous
for $\mathcal{R}$} if:
\begin{description}
\item [{UB1}] 
given a $C^{*}$-algebra $A,$ if $\varphi:U\rightarrow A$
is a continuous $*$-homomorphism then
$\varphi\circ\iota:\mathcal{X}\rightarrow A$
is a representation in $\mathcal{R};$ 
\item [{UB2}] 
given a $C^{*}$-algebra $A,$ if a
function $f:\mathcal{X}\rightarrow A$
is a representation in $\mathcal{R}$ then there is a unique continuous
$*$-homomorphism $\varphi:U\rightarrow A$ so that $f=\varphi\circ\iota.$
\end{description}
\end{defn}

\begin{lem}
\label{lem:ubiquitousIFFuniversal} Every $C^{*}$-relation
$\mathcal{R}$
has an ubiquitous function, namely the universal representation of
the extension $\hat{\mathcal{R}}$ of $\mathcal{R}$ 
to a pro-$C^{*}$-relation.
\end{lem}

\begin{proof}
Proposition~\ref{pro:ExtendingCstarRelationsToProCstar} assures
us that $\hat{\mathcal{R}}$ exists. 
Consider the universal representation
$\iota:\mathcal{X}\rightarrow U$ of 
$\hat{\mathcal{R}}.$ Suppose
$A$ is a $C^{*}$-algebra. If $\varphi:U\rightarrow A$ 
is a continuous
$*$-homomorphism then $\varphi\circ\iota$ is in $\hat{\mathcal{R}}$
and so in $\mathcal{R}.$ If a function $f:\mathcal{X}\rightarrow A$
is a representation in $\mathcal{R}$ then it is a representation
in $\hat{\mathcal{R}},$ so there is a unique continuous 
$*$-homomorphism
$\varphi:U\rightarrow A$ so that $f=\varphi\circ\iota.$
\end{proof}

\begin{lem}
The ubiquitous function for a $C^{*}$-relation is unique.
\end{lem}

\begin{proof}
We will show that a function $f:\mathcal{X}\rightarrow U$ that is
ubiquitous for $\mathcal{R}$ is universal for $\hat{\mathcal{R}}.$

Suppose $f:\mathcal{X}\rightarrow A$ is a representation 
of $\hat{\mathcal{R}}.$
Then for all $p$ in $S(A),$ the composition $\pi_{p}\circ f$ is
a representation of $\mathcal{R}.$ For each $p$ there is a unique
continuous $*$-homomorphism $\varphi_{p}:U\rightarrow A_{p}$ so that 
$\varphi_{p}\circ\iota=\pi_{p}\circ f.$
If $p^{\prime}\geq p$ then
\[
\pi_{p^{\prime},p}\circ\varphi_{p^{\prime}}\circ\iota = 
\pi_{p^{\prime}}\circ f
\]
and so, by uniqueness, 
$\pi_{p^{\prime},p}\circ\varphi_{p^{\prime}}=\pi_{p}.$
There is, therefore, a unique continuous $*$-homomorphism 
$\varphi:U\rightarrow A$
such that $\pi_{p}\circ\varphi=\varphi_{p}.$ 
Therefore $\pi_{p}\circ\varphi\circ\iota=\pi_{p}\circ f$
for all $p,$ and so $\pi_{p}\circ\varphi=f.$

If $\varphi^{\prime}\circ\iota=f$ 
then $\pi_{p}\circ\varphi^{\prime}\circ\iota=\pi_{p}\circ f,$
and so by the uniqueness of the $\varphi_{p}$ we 
have $\pi_{p}\circ\varphi^{\prime}=\varphi_{p}.$
Therefore $\pi_{p}\circ\varphi^{\prime}=\pi_{p}\circ\varphi$ for
all $p,$ and so $\varphi^{\prime}=\varphi.$
\end{proof}

\begin{thm}
Suppose $\mathcal{X}$ is finite. If $\mathcal{R}$ is a closed $C^{*}$-relation
on $\mathcal{X}$ then there exists a function $\iota:\mathcal{X}\rightarrow U$
such that:
\begin{enumerate}
\item 
$\iota$ is ubiquitous for $\mathcal{R}$ 
and $U$ is a $\sigma$-$C^{*}$-algebra;
\item 
the induced continuous $*$-homomorphism 
$\bar{\iota}:\mathbb{F}\langle\mathcal{X}\rangle\rightarrow U$
is onto and induces an 
isomorphism $U\cong\mathbb{F}\langle\mathcal{X}\rangle/I$
for $I=\ker(\bar{\iota});$
\item 
there is a single element $g$ of $\mathbb{F}\langle\mathcal{X}\rangle$
so that
\[
U \cong 
\textnormal{\textnormal{pro}}C^{*}\left\langle
\mathcal{X}\left|g(x_{1},\ldots,x_{n})=0\right.
\right\rangle .
\]

\end{enumerate}
\end{thm}

\begin{proof}
Let $\mathcal{S}_{n}$ denote the $C^{*}$-relations
\[
\| x\|\leq n\quad(\forall x\in\mathcal{X}).
\]
Then $\mathcal{S}_{n}$ and $\mathcal{S}_{n}\cap\mathcal{R}$ are
compact. We get a commutative diagram 
\[
\xymatrix{
	C^* \left \langle 
	\mathcal {X}
	\left |
	\mathcal {S} _{n+1}
\right. \right \rangle 
	\ar[r] \ar[d]
	&
	C^* \left \langle 
	\mathcal {X}
	\left |
	\mathcal {S} _{n+1} \cap \mathcal {R} 
	\right. \right \rangle 
	\ar[d]
\\
	C^* \left \langle 
	\mathcal {X}
	\left |
	\mathcal {S} _{n}
	\right. \right \rangle 
	\ar[r]
	&
	C^* \left \langle 
	\mathcal {X}
	\left |
	\mathcal {S} _{n} \cap \mathcal {R} 
	\right. \right \rangle 
}
\]
where
all the maps are induced by the identity on the generators. This is
clearly a pushout with surjective $*$-homomorphisms. Let $\mathcal{U}$
be the $\sigma$-$C^{*}$-algebra
\[
U = 
\lim_{\longleftarrow}
C^{*}\left\langle
\mathcal{X}\left|\mathcal{S}_{n}\cap\mathcal{R}\right.
\right\rangle 
\]
and let $\iota:\mathcal{X}\rightarrow U$ denote the limit of the
$\iota_{n}=\iota_{\mathcal{S}_{n}\cap\mathcal{R}}.$ 
Theorem~\ref{thm:inverseLimitsOfPushouts}
applies, telling us that 
$\bar{\iota} :
\mathbb{F}\langle\mathcal{X}\rangle\rightarrow U$
is onto.

Suppose $A$ is a $C^{*}$-algebra and $\varphi:U\rightarrow A$ is
a continuous $*$-homomorphism. By 
Lemma~\ref{lem:mapsProCstarToCstarFactor},
for some $n$ there is a $*$-homomorphism 
\[
\bar{\varphi} :
C^{*}\left\langle \mathcal{X}\left|\mathcal{S}_{n}\cap R\right.\right\rangle 
\rightarrow A
\]
so that $\varphi=\bar{\varphi}\circ\rho_{n}.$ 
This means that $\varphi\circ\iota=\bar{\varphi}\circ\iota_{n}$
is a representation of $\mathcal{R}.$ 

Given a $C^{*}$-algebra $A$ and a representation 
$f:\mathcal{X}\rightarrow A,$
for some $n$ we have $\| f(x)\|\leq n$ for all $x$ in $\mathcal{X}$
and so have a $*$-homomorphism
\[
\varphi_{n} :
C^{*}\left\langle 
\mathcal{X}\left|\mathcal{S}_{n}\cap R\right.
\right\rangle
\rightarrow A
\]
for which 
$f=\varphi_{n}\circ\iota_{\mathcal{S}_{n}\cap R}.$ Therefore
$f=\left(\varphi_{n}\circ\rho_{n}\right)\circ\iota.$
Uniqueness follows
since $\iota(\mathcal{X})$ generates $U.$ 

By \cite[Corollary 5.4]{PhillipsInverseLimits} 
we have an isomorphism
$U\cong\mathbb{F}\langle\mathcal{X}\rangle/I$ for 
$I=\ker(\bar{\iota}).$

To prove (3) we modify a technique from 
\cite[Theorem 2.1]{EilersLoringContingenciesStableRelations}
and \cite[Proposition 41]{HadwinContinuousFunctions}. Certainly
\[
U \cong
\textnormal{pro}C^{*}\left\langle
\mathcal{X}\left|\ g(x_{1},\ldots,x_{n})=0\ (\forall g\in I)\right.
\right\rangle .
\]
By the separability of 
$\mathcal{F}\langle\mathcal{X}\rangle$ we
may replace all the elements of $I$ with a sequence so that
\[
U \cong 
\textnormal{pro}C^{*}\left\langle 
\mathcal{X}\left|\ g_{k}(x_{1},\ldots,x_{n})=0\ (\forall k\in\mathbb{N})\right.
\right\rangle .
\]
The fact that $y^{*}y=0$ in a $C^{*}$-algebra if and only if $y=0$
allows us to replace the $g_{k}$ as needed to ensure the $g_{k}$
are positive elements in $I.$ Let $p_{n}$ be a sequence 
of $C^{*}$-seminorms
defining the topology on $I.$ Taking a sequence of positive scalars
$\alpha_{k}$ so that $\alpha_{k}\leq\left(2^{k}p_{r}(g_{k})\right)^{-1}$
for $1\leq r\leq k$ we can ensure that $g=\sum\alpha_{k}g_{k}$ exists,
and then 
\[
U \cong
\textnormal{pro}C^{*}
\left\langle \mathcal{X}\left|g(x_{1},\ldots,x_{n})=0\right.
\right\rangle .
\]

\end{proof}

\section{Matrices having $C^{*}$-Relations}

In studying the boundary maps in $K$-theory 
(\cite{LoringProjectiveKtheory,LoringProjCstarAndBoundary})
we proved the projectivity of the $C^{*}$-algebras

\begin{equation}
C^{*}\left\langle h,k,x
\,\left|\,
P^{2}=P^{*}=P\mbox{ for }P=\left[\begin{array}{cc}
\mathbb{1}-h & x^{*}\\
x & k\end{array}\right]\right.\right\rangle \label{eq:matrixProjection}
\end{equation}
and
\begin{equation}
C^{*}\left\langle
h,k,x
\,\left|\,
hk=0\mbox{ and }0\leq P\leq1\mbox{ for }P=\left[\begin{array}{cc}
\mathbb{1}-h & x^{*}\\
x & k\end{array}\right]\right.
\right\rangle 
\label{eq:matrixPositiveFancy}
\end{equation}
and, implicitly at least, also
\begin{equation}
C^{*}\left\langle
h,k,x
\,\left|\,
0\leq P\leq1\mbox{ for }P=\left[\begin{array}{cc}
\mathbb{1}-h & x^{*}\\
x & k\end{array}\right]\right.
\right\rangle .
\label{eq:MatrixPositivePlain}
\end{equation}
 We use $\tilde {A}$ to denote
the unitization of a $C^*$-algebra $A,$ where the unit $\mathbb{1}$
is adjoined even
when $A$ is unital. It may not be obvious
these $C^{*}$-algebras exist. They do, and
there is a general method to reinterpret $C^{*}$-relations
in $\mathbf{M}_{n}(\tilde{B})$
as $C^{*}$-relations in $B.$

We are adding a chapter to an old story whose beginnings 
include \cite{Brown-free_product}
by Bergman and \cite{BergmanCoproducts} by Larry Brown. In the nonunital
case, we cannot use a trick with free products and relative commutants.
We must face the universal nonsense. 

In this section $n$ is a positive integer. 

\begin{notation}
Let $\overline{n}=\{1,2,\ldots,n\}.$ 
\end{notation}
\begin{defn}
\label{def:ShiftedComatrixRelations} Suppose $\mathcal{R}$ is a
$C^{*}$-relation on $\mathcal{X}$ and that 
$\alpha:\mathcal{X}\rightarrow\mathbf{M}_{n}(\mathbb{C})$
is a representation of $\mathcal{R}.$ Define $\mathcal{R}_{\alpha}$
as the full subcategory of
$\mathcal{F}_{\mathcal{X}\times\overline{n}\times\overline{n}}$
whose objects are the functions
\[
f:\mathcal{X}\times\overline{n}\times\overline{n}\rightarrow B
\]
for which 
$f_{\alpha} :\mathcal{X}\rightarrow\mathbf{M}_{n}(\widetilde{B})$
is a representation of $\mathcal{R},$ where
\[
f_{\alpha}(x) =
\sum_{i,j}\left(\alpha_{ij}\mathbb{1}+f(x,i,j)\right)\otimes e_{ij}.
\]

\end{defn}
For example, in (\ref{eq:matrixProjection}) $\mathcal{R}$ is the
relation $p^{2}=p^{*}=p$ and 
\[
\alpha(p)=\left[\begin{array}{cc}
1 & 0\\
0 & 0\end{array}\right].
\]
The generator $(p,1,2)$ is redundant, 
$h= -(p,1,1),$ $k= (p,2,2)$ and
$x= (p,2,1).$

\begin{lem}
With $\mathcal{R}$ and $\alpha$ as 
in Definition~\ref{def:ShiftedComatrixRelations},
$\mathcal{R}_{\alpha}$ is a $C^{*}$-relation 
on $\mathcal{X}\times\overline{n}\times\overline{n}.$
It is compact when $\mathcal{R}$ is compact. 
It is closed when $\mathcal{R}$
is closed. 
\end{lem}

\begin{proof}
Suppose first that $\mathcal{R}$ is any $C^{*}$-relation 
on $\mathcal{X}.$

If $f$ is the zero map 
\[
f :
\mathcal{X}\times\overline{n}\times\overline{n}\rightarrow\{0\}
\]
then $f_{\alpha}=\alpha$ is a representation of $\mathcal{R}$ so
$f$ is a representation of $\mathcal{R}_{\alpha}.$

Suppose $\varphi:A\rightarrow B$ is an injective $*$-homomorphism
and 
\[
\varphi\circ f: 
\mathcal{X}\times\overline{n}\times\overline{n}\rightarrow B
\]
is a representation in $\mathcal{R}_{\alpha}.$ 
Then 
$\mathbf{M}_{n}\left(\widetilde{\varphi}\right) =
\widetilde{\varphi}\otimes\mbox{id}$
is also an injective $*$-homomorphism,
\[
\mathbf{M}_{n}\left(\widetilde{\varphi}\right) :
\mathbf{M}_{n}
\left(\widetilde{A}\right)\rightarrow\mathbf{M}_{n}\left(\widetilde{B}\right)
\]
and
\begin{align*}
(\varphi\circ f)_{\alpha}(x) 
& =  
\sum_{i,j}\left(\alpha_{ij}\mathbb{1}+\varphi(f(x,i,j))\right)\otimes e_{ij}\\
 & =  
 \mathbf{M}_{n}
 \left(\widetilde{\varphi}\right)
 \left(\sum_{i,j}\left(\alpha_{ij}\mathbb{1}+f(x,i,j)\right)
 \otimes e_{ij}\right)\\
 & =  
 \left(\mathbf{M}_{n}
 \left(\widetilde{\varphi}\right)\circ f_{\alpha}\right)(x).
\end{align*}
Since $\varphi\circ f$ is a representation of $\mathcal{R}$, we
know $\mathbf{M}_{n}\left(\widetilde{\varphi}\right)\circ f_{\alpha}$
is a representation of $\mathcal{R}_{\alpha}.$ Therefore $f_{\alpha}$
is a representation of $\mathcal{R}_{\alpha}$ and so $f$ is a representation
of $\mathcal{R}.$

Suppose $\varphi:A\rightarrow B$ is a $*$-homomorphism and 
\[
f:\mathcal{X}\times\overline{n}\times\overline{n}\rightarrow A
\]
is a representation in $\mathcal{R}_{\alpha}.$ Then we still have
that $\mathbf{M}_{n}\left(\widetilde{\varphi}\right)$ is a 
$*$-homomorphism
and
\[
(\varphi\circ f)_{\alpha} =
\mathbf{M}_{n}\left(\widetilde{\varphi}\right)\circ f_{\alpha}.
\]
Since $f$ is a representation, so is $f_{\alpha}$. Therefore 
$(\varphi\circ f)_{\alpha}$
is a representation, and so $\varphi\circ f$ is a representation.

Now suppose 
\[
f_{\lambda} : 
\mathcal{X}\times\overline{n}\times\overline{n}\rightarrow A_{\lambda}
\]
is a representation for each $\lambda$ in a nonempty, finite set
$\Lambda.$ Each $\left(f_{\lambda}\right)_{\alpha}$ is a representation.
Let 
\[
\Phi :
\mathbf{M}_{n}
\left(\left(\prod_{\lambda}A_{\lambda}\right)^{\sim}\right)
\rightarrow
\prod_{\lambda}\mathbf{M}_{n}\left(\widetilde{A_{\lambda}}\right)
\]
be the injective $*$-homomorphism defined by
\[
\Phi\left(
\left(\beta\mathbb{1}+\left\langle a_{\lambda}\right\rangle _{\lambda}\right)
\otimes e_{ij}
\right)
=
\left\langle \left(\beta\mathbb{1}+a_{\lambda}\right)
\otimes e_{ij}\right\rangle _{\lambda}.
\]
Since
\[
\Phi\circ\left(\prod_{\lambda}f_{\lambda}\right)_{\alpha}
=\prod_{\lambda}\left(f_{\lambda}\right)_{\alpha}
\]
we know that
\[
\left(\prod_{\lambda}f_{\lambda}\right)_{\alpha}
\]
is a representation. This means $\prod_{\lambda}f_{\lambda}$ is a
representation.

If $\mathcal{R}$ is compact, then the above argument works for infinite
sets $\Lambda.$ If $\mathcal{R}$ is only closed, we need to add
the assumptions
\[
\sup_{\lambda}\left\Vert f_{\lambda}(x,i,j)\right\Vert <\infty
\]
for each $x$ and each $i$ and $j.$ This forces, for each $x,$
\[
\sup_{\lambda}\left\Vert \left(f_{\lambda}\right)_{\alpha}(x)\right\Vert 
=
\sup_{\lambda}
\left\Vert \sum_{i,j}\left(\alpha_{ij}\mathbb{1}
+
f_{\lambda}(x,i,j)\right)\otimes e_{ij}\right\Vert 
<\infty
\]
and the above argument is still fine.
\end{proof}

This is helpful even when $n$ is $1.$ For example there is 
\[
C_{0}(0,1)=C^{*}\left\langle x\left|\strut(\mathbb{1}+x)^{*} =
(\mathbb{1}+x)^{-1}\right.\right\rangle .
\]
For an example that does not produce a $C^{*}$-algebra, there is
\[
\textnormal{pro}C^{*}\left\langle a,b,c,d
\,\left|\,
P^{2} =
P\mbox{ for }P=\left[\begin{array}{cc}
\mathbb{1}+a & b\\
c & d\end{array}\right]\right.\right\rangle .
\]
In these two examples it is easy rewrite the relations as $*$-polynomials
not involving matrices. Such a reduction is not always practical,
as illustrated by
\[
C^{*}\left\langle
a,b,c,d
\,\left|\,
0\leq P\leq1\mbox{ for }P=\left[\begin{array}{cc}
a & b\\
c & d\end{array}\right]\right.
\right\rangle .
\]

Define 
$\lambda :
\widetilde{A}\rightarrow\mathbb{C}$ by $\lambda(\alpha\mathbb{1}+a)=\alpha.$

\begin{defn}
If $A$ is a $C^{*}$-algebras and 
$\alpha:A\rightarrow\mathbf{M}_{n}(\mathbb{C})$
is a $*$-homomorphism, define $\mathbf{W}_{\alpha}(A)$ as 
\[
C^{*}\left\langle 
A\times\overline{n}\times\overline{n}
\,\left|\,\strut
a\mapsto\left[\strut\alpha_{ij}\mathbb{1}+(a,i,j)\right]
\right.
\mbox{ is a $*$-homomorphism}
\right\rangle .
\]
\end{defn}

That is, $\mathbf{W}_{\alpha}(A)$ has a $*$-homomorphism
$\iota:A\rightarrow\mathbf{M}_{n}(\widetilde{A})$
so that $\mathbf{M}_{n}(\lambda)\circ\iota=\alpha$ that is universal
for all $*$-homomorphism $\iota:A\rightarrow\mathbf{M}_{n}(\widetilde{B})$
such that $\mathbf{M}_{n}(\lambda)\circ\iota=\alpha.$ If $\alpha=0$
then
\[
\hom(\mathbf{W}_{\alpha}(A),B)\cong\hom(A,\mathbf{M}_{n}(B))
\]
 and $\mathbf{W}_{\alpha}=\mathbf{W}_{n}$ is the left-adjoint to
the functor $\mathbf{M}_{n}$ and was investigated by Phillips in
\cite{PhillipsClassifyingAlgebras}.

\begin{thm}
\label{thm:coMatrixStaysProjective}If $A$ is projective and 
$\alpha:A\rightarrow\mathbf{M}_{n}(\mathbb{C})$
is a representation then $\mathbf{W}_{\alpha}(A)$ is projective.
\end{thm}
\begin{proof}
Suppose we have a diagram 
\[
\xymatrix{
	&	B \ar@{->>}[d] ^{\rho} \\
\mathbf {W}_{\alpha}(A) \ar[r] ^(0.6){\varphi} \ar@{-->}[ru] ^{\bar {\varphi}}
	&  D 
}
\]
in which $\rho$ is
surjective, $\varphi$ is given and we want to find $\bar{\varphi}$
making the diagram commute. This translates to the lifting problem
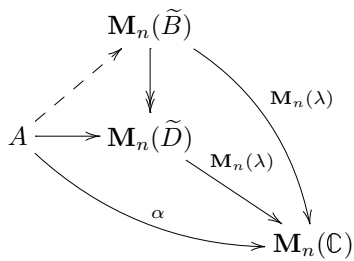
\[
\xymatrix{
	&	\mathbf {M}_{n}(\widetilde {B}) 
		\ar@{->>}[d] \ar@/^0.5cm/[rdd]^{ \mathbf {M}_{n}(\lambda)} \\
A \ar[r] \ar@/_0.5cm/[drr] ^{\alpha} \ar@{-->}[ru] 
	&  \mathbf {M}_{n}(\widetilde {D})  
		\ar[rd]^(0.4){ \mathbf {M}_{n}(\lambda)} \\
		&& \mathbf {M}_{n}(\mathbb{C})
}
\]
which is easily solved.
\end{proof}

For example, Theorem~\ref{thm:coMatrixStaysProjective} tells us
that
\[
C^{*}\left\langle a,b,c,d
\,\left|\,
\left\Vert \left[\begin{array}{cc}
a & \mathbb{1}+b\\
c & d\end{array}\right]\right\Vert \leq1\right.\right\rangle 
\]
is projective.

\providecommand{\bysame}{\leavevmode\hbox to3em{\hrulefill}\thinspace}

\end{document}